\documentclass[a4paper,12pt]{article}
\usepackage{amsthm}
\usepackage{amssymb}
\usepackage{amsmath}
\usepackage{amsfonts}
\usepackage{hyperref}
\usepackage{graphicx}
\newtheorem{tw}{Theorem}[section]
\newtheorem{stw}[tw]{Proposition}
\newtheorem{lem}[tw]{Lemma}
\newtheorem{wn}[tw]{Corollary}
\theoremstyle{definition}
\newtheorem{df}[tw]{Definition}
\newtheorem{prz}[tw]{Example}
\newtheorem{uw}[tw]{Remark}
\newtheorem{fa}[tw]{}
\theoremstyle{plain}

\begin{document}
\title{Umbilical routes along geodesics and hypercycles in the hyperbolic space}
\author{
Maciej Czarnecki 
\\Uniwersytet \L\'odzki, \L\'od{\'z}, Poland \\
Wydzia\l\ Matematyki i Informatyki, Katedra Geometrii\\
e-mail: maczar@math.uni.lodz.pl
}

\maketitle

\begin{footnote}{{\it 2010 Mathematics Subject Classification.} 53C12, 53A30, 53A35.}
\end{footnote}

\abstract{Given a geodesic line $\gamma$ the hyperbolic space $\mathbb H^n$ we formulate a necessary and sufficient condition for a function along this geodesic which measure the mean curvature of totally umbilical leaves of a foliation orthogonal to $\gamma$. Then we extend the result to $\gamma$ being a hypercycle i.e. a geodesic on a hypersurface equidistant from the totally geodesic one.}

\section*{Introduction}

In the geometric theory of foliations, a   question on foliations with totally umbilical leaves comes just after that on totally geodesic foliations. The last one for compact or finite volume manifolds has definite and negative answer (see \cite{CzW} for some history). In \cite{LW} Langevin and Walczak proved that on constant curvature closed manifold there are no totally umbilical foliations. For open manifolds there are geometrical classifications of totally geodesic foliations in the hyperbolic space by Ferus (\cite{F}) and Browne (\cite{Br}).

The question on totally umbilical routes along curves in the real hyperbolic space $\mathbb H^n$ was formulated in \cite{CzW}. In the paper presented (and announced in \cite{BaCz}), we give some partial answer restricting to transversals which are geodesics, horocycles, or hypercycles. More general result was obtained by the author and Langevin in \cite{CzL} --- see \ref{shadok} for a mention.

The most general result of the paper is Theorem \ref{route_hyperc} giving necessary and sufficient condition for a function along arc--length paramet\-rized hypercycle to generate totally umbilical foliations. Namely, if the hypercycle has (constant) geodesic curvature $\cos\varphi$ then this condition states that the mean curvature of leaves $h$ starts for some period from $-\sin\varphi$, ends from some time with $\sin\varphi$ and between some modified hyperbolic arcus tangent of $h$ of leaves is $(\sin\varphi )$--Lipschitz function. The condition is more visible in case of a geodesic transversal (Theorem \ref{route_geod}): $-{\rm ath}\circ h$ is $1$--Lipschitz.
Respective inequalities for differentiable $h$ appear in Theorem \ref{route_hyperc_c1} and 
Theorem \ref{route_geod_c1}. The last one is very simple $h'\ge h^2-1$ and shows a potential of change geometry: bigger if not far from totally geodesic.  

\section{Umbilical hypersurfaces of the hyperbolic space}

Umbilicity is a standard notion in Riemannian geometry and one of the easiest which is conformally invariant.

A point on a submanifold of a Riemannian manifold is called \emph{umbilical} if 
all eigenvalues of the shape operator at this point are equal. 
In this case every such eigenvalue equals the mean curvature up to sign.
Consequently, a submanifold is \emph{totally umbilical} if consists only of umbilical points and a \emph{totally umbilical foliation} of a Riemannian manifold is 
a foliation with all the leaves totally umblical.

\medskip

For the real $n$--dimensional hyperbolic space $\mathbb H^n$ consider its half--space model i.e. the
 set $\Pi^{n,+}=\{ x\in\mathbb R^n\ | \ x_n>0\}$
endowed with the Riemannian metric
$$g(X,Y)_x=\frac{1}{x_n^2}\langle X,Y\rangle$$
where $\langle .,.\rangle$ denote the standard Euclidean inner product. 

The hyperbolic distance in the half--space is given by the formula (cf. \cite{Be-Pe})
$$
d(x,y)= 2\,{\rm ath}\, \sqrt{\frac{\| \hat x-\hat y\|^2+(x_n-y_n)^2}{\| \hat x-\hat y\|^2+(x_n+y_n)^2}}
$$
where $\hat x=(x_1,\ldots ,x_{n-1})$ and analogously for $y$. Here
$${\rm ath} =(\tanh )^{-1}\ :\ t\mapsto \ln\sqrt{\frac{1+t}{1-t}}.$$
In the particular case $\Pi^{2,+}\subset \mathbb C$, 
$$d(z,w)=2\,{\rm ath}\, \left| \frac{z-w}{z-\bar w}\right|,$$
especially $d(ai,bi)=\left| \ln\frac{a}{b}\right|$ for $a,b>0$.

\medskip

Every isometry of the half--space model is a conformal diffeomorphism $\Pi^{n,+}$ on itself i.e. a composition of a horizontal translation, inversion in a sphere orthogonal to the ideal boundary or identity, and orthogonal transformation in the first $n-1$ variables (cf.
\cite{Be-Pe}).

In particular, for any two geodesic lines there is an isometry sending one to another. 

\medskip

$\mathbb H^n$ is an Hadamard manifold so in the purely metric way (cf. \cite{BH}) we could define horospheres and the ideal boundary. In the half--space model, a \emph{horosphere} is a sphere tangent to 
$\mathbb R^{n-1}\times \{ 0\}$ (without tangency point) or a hyperplane parallel to it. The \emph{ideal boundary} is a topological
$(n-1)$ sphere  $\left( \mathbb R^{n-1}\times \{ 0\} \right)\cup \{ \infty\}$.

\emph{Totally gedesic hypersurfaces} are open hemi--spheres or open half--hyperplanes orthogonal to $\mathbb R^{n-1}\times \{ 0\}$. A connected component of a set equidistant from a  
totally gedesic hypersurface is called a \emph{hypersphere}. In the half--space model,
a hypersphere is a part of a sphere or hyperplane transversely intersecting $\mathbb R^{n-1}\times \{ 0\}$
included in $\Pi^{n,+}$.

\begin{df}
We use the common name \emph{generalized hypersphere} for a complete hypersurface in $\mathbb H^n$
which is either horosphere, hypersphere or totally geodesic and attach to it its angle of intesection with the ideal boundary.

Thus a horosphere is $0$-or-$\pi$--hypersphere (depending on its end) while a totally geodesic hypersurface is a $\frac{\pi}{2}$--hypersphere.
\end{df}

\begin{stw}\label{curv}
Using orientation inside a generalized hypersphere, i.e. in the half--space "down", we observe that a hypersphere making angle $\beta$ (measured outside) with the ideal boundary has constant mean curvature $h=-\cos\beta$; this includes a horizontal horosphere of $h\equiv 1$.
\end{stw}

\begin{proof}
Following calculation of Christoffel symbols for a conformal change of Riemannian metric (\cite{dC2}) the second fundamental form for hyperplanes in $\Pi^{n,+}$ is easy to extract. In \cite{Lu}, Lu\.zy\'nczyk observed that the shape operator is the identity multiplied by the last coordinate of the normal vector (unit in the Euclidean norm). 
\end{proof}

\begin{stw} \label{umb_hyp}
A connected complete unbounded hypersurface of $\mathbb H^n$ is totally umbilical iff it is a generalized hypersphere.
\end{stw}

\begin{proof}
It is classical (cf. \cite{dC1} in case $n=3$) that any totally umbilical hypersurface of $\mathbb R^n$ is contained in a sphere or in a hyperplane. 

The half--space model of $\mathbb H^n$ is conformally equivalent to $\mathbb R^n$. Conformal diffeomorphisms preserve umbilicity hence all connected complete totally umbilical hypersurfaces in the half--space model are nonempty intersections of $\Pi^{n+}$ by a sphere or
a hyperplane.

Among them there are metric spheres which are bounded so any unbounded complete umbilical hypersurface is the cross--section of a sphere or hyperplane not disjoint with ideal boundary i.e. a generalized hypersphere.
\end{proof}

\begin{df}
A \emph{$\varphi$--hypercycle} is a geodesic line on a $\varphi$--hypersphere, $\varphi\in [0,\pi]$.
In the half--space a $\varphi$--hypercycle is a cross--section of a $\varphi$--hypersphere with a $2$--dimensional plane through its center or simply open ray making angle $\varphi$ with teh ideal boundary.
\end{df}

\begin{prz}
A generalized $\varphi$--hypercycle has constant geodesic curvature equal $|\cos\varphi|$ (cf. \cite{Cz1})
In $\Pi^{2,+}$ hypercycles (at the same time hyperspheres, $n=2$) are
\begin{enumerate}
\item geodesic ($\varphi =\frac{\pi}{2}$) of ideal ends $0$ and $\infty$ being positive imaginary half--axis $i\mathbb R_+$ para\-metrized by arc--length as $t\mapsto ie^t$.
\item $\varphi$--hypercycle $E_{\varphi} =e^{i\varphi} \mathbb R_+$, $\varphi\in \left(0,\frac{\pi}{2}\right)$ of ideal ends $0$ and $\infty$ is parametrized by arc--length as $t\mapsto e^{t\sin\varphi +i\varphi}$.
\item horospheres ($\varphi =\pi$) with  the ideal end $\infty$ have arc--length parametrizations $t\mapsto t+ia$ with $a>0$.
\end{enumerate}
\end{prz}

\section{Umbilical routes along geodesics}

The notion of umbilical route says how to change an umbilicity parameter (mean curvature which is equal to the eigevalue of the shape operator) to preserve nice location of a family of umbilical hypersurfaces and avoid intersections.
\begin{df}
Let $\gamma :\mathbb R\to \mathbb H^n$ be an arc--length parametrized curve. We say that a real function $h$ is an \emph{umbilical route} along $\gamma$ if the family $L_t$ of generalized hyperspheres
orthogonal to $\gamma$ and having mean curvature $h(t)$ at $\gamma (t)$ could be extended to a totally umbilical foliation of $\mathbb H^n$.
\end{df}

In codimension $1$ the real hyperbolic space is the only carrying nonotrivial umbilical routes. Other constant curvature space i.e. $\mathbb R^n$ and $\mathbb S^n$ due to their structures of totally umbilical complete hypersurfaces (full spheres or hyperplanes) have topological obstrucions for existence of totally umbilical foliations ---
interior of any spherical leaf cannot foliated. On the other hand, any two nonparallel hyperplanes in $\mathbb R^n$ intersect. Hence for any curve in $\mathbb S^n$ there are no umbilical routes while in $\mathbb R^n$ the only identically zero appears on straight lines.

In nonconstant curvature even very regular symmetric space like the complex hyperbolic space $\mathbb CH^n$ have no totally umbilical hypersurafces. 

\medskip

We start with a very simple case of umbilical route along a geodesic where the situation is clear and formulae are predictable.

\begin{tw}
\label{route_geod}
Let $\mathcal F$ be a transversely $C^0$ codimension $1$ totally umbilical foliation of $\mathbb H^n$ orthogonal to an  arc--length parametrized geodesic line $\gamma$. If for any $t\in\mathbb R$ the mean curvature of the leaf (taken with orientation opposite to $\gamma$) through $\gamma (t)$ is $h(t)$ then $|h|\le 1$ 
and there are $ t_-, t_+\in [-\infty,+\infty]$ such that
\begin{equation}\label{eq:h_geod}
\begin{array}{rl}
({\rm i})   & h|_{(-\infty ,t_-]}\equiv -1, \\
({\rm ii}) & {the\ function}\  (-{\rm ath} \circ h)|_{(t_- ,t_+)}\  {\rm is}\  1-{Lipschitz}, \\
({\rm iii})  & h|_{[t_+,+\infty )}\equiv 1.
\end{array}
\end{equation}
Conversely, if $h:\mathbb R\to [-1,1]$ is a continuous function satisfying {\rm (\ref{eq:h_geod})} then $h$ is an umbilical route along any geodesic line in $\mathbb H^n$. 
\end{tw}

For the proof we need elementary lemmas.

\begin{lem}\label{circ_geod}
Let $s>0$, $\beta\in [0,\pi )$ and $\mathcal C$ be a circle of center $C\subset\mathbb C$ and radius $R$ orthogonal to the imaginary axis $i\mathbb R$ at the point $is$ and meeting the real axis $\mathbb R$ at angle $\beta$ (measured outside).

Then 
$$R=\frac{s}{1+\cos\beta}, \quad C=i\frac{s \cos\beta}{1+\cos\beta}$$
and the point(s) of the intersection $\mathcal C\cap \mathbb R$ are of the form 
$$a_{\pm}=\pm s\tan\frac{\beta}{2}.$$
\end{lem}

\begin{figure}[h]
\centerline{
\includegraphics[scale=1]{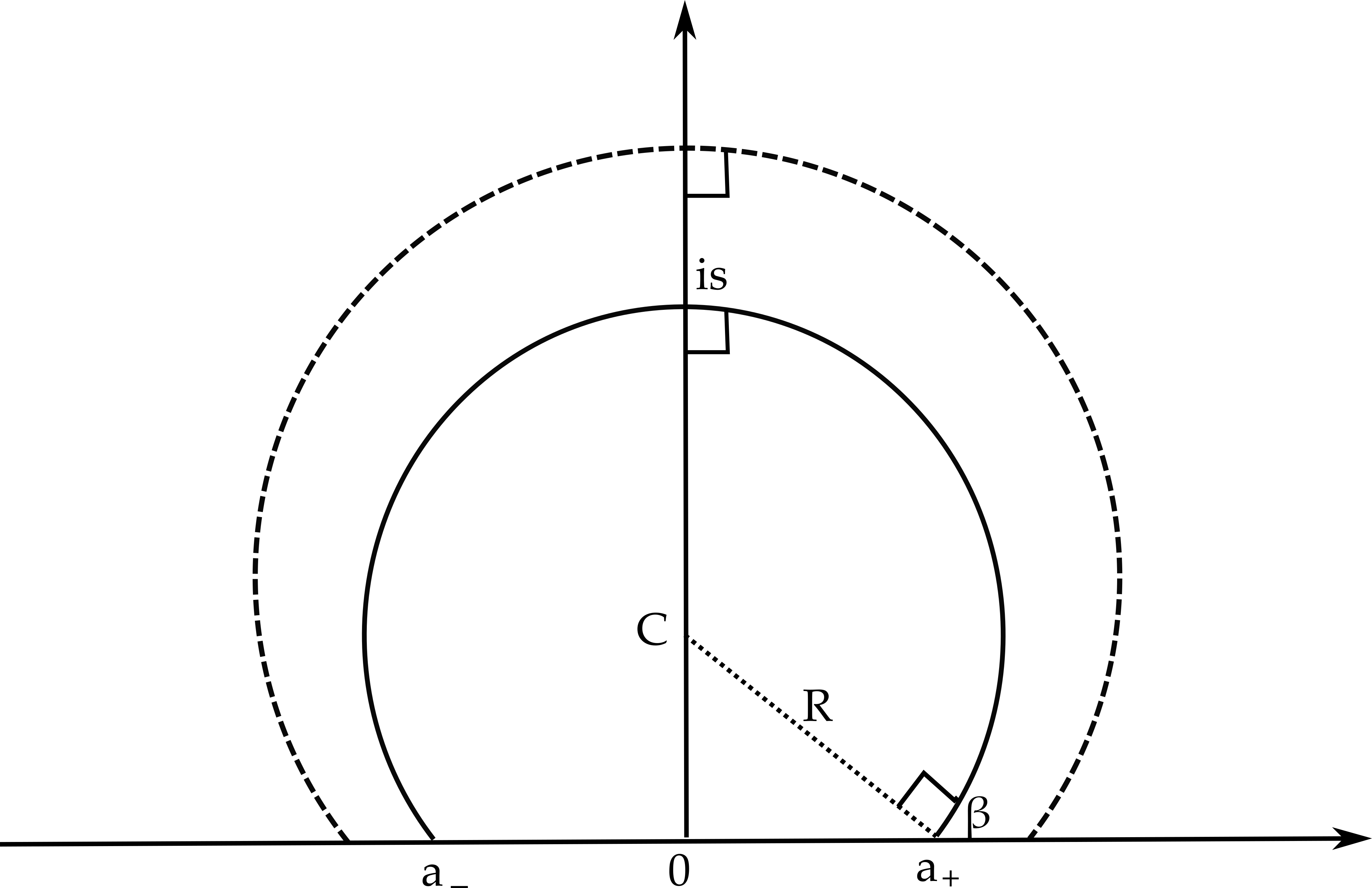}
}
\caption{Hypercycles orthogonal to a geodesic}
\label{fig:fig1}
\end{figure}

\begin{proof}
Since $\mathcal C$ is orthogonal to $i\mathbb R$ its center $C\in i\mathbb R$ and $\Im C=s-R$. At a point
$a\in \mathcal C \cap \mathbb R$ radius is perpendicular to the tangent. If $\beta$ is acute (other cases are similar) then 
$\sphericalangle 0aC =\frac{\pi}{2}-\beta$ and 
$$\frac{s-R}{R}=\sin \left(\frac{\pi}{2}-\beta \right) \ {\rm which\ implies} \ R=\frac{s}{1+\cos\beta}.$$
Thus we have $C$ and
$$a=\pm R\sin\beta =\pm s\frac{\sin\beta}{1+\cos\beta}=\pm s\tan\frac{\beta}{2}.$$
\end{proof}

\begin{lem}\label{circs_geod}
Let $0<s_1<s_2$, $\beta_1, \beta_2\in [0,\pi )$ and $\mathcal C_1, \mathcal C_2\subset\mathbb C$ be circles  orthogonal to the imaginary axis $i\mathbb R$ at points $is_1$, $is_2$ and meeting the real axis $\mathbb R$ at angles $\beta_1$, $\beta_2$, respectively.

Then $\mathcal C_1$ and $\mathcal C_2$ do not intersect in the upper half--plane $\Pi^{2,+}$ iff
$$ \frac{\cot\frac{\beta_2}{2}}{\cot\frac{\beta_1}{2}}\le \frac{s_2}{s_1}$$
provided that $\beta_1, \beta_2\in (0,\pi )$ or $\beta_1=\beta_2=0$.
\end{lem}

\begin{proof}
For a given $\mathcal C_1$ which intersects $\mathbb R$ transversally ($\beta_1 >0$) the only situation of
$\mathcal C_1\cap\mathcal C_2\cap \Pi^{2,+}=\emptyset$ is that $a_{2-}\le a_{1-}$ and $a_{1+}\le a_{2+}$.
Hence we obtain the inequality by Lemma \ref{circ_geod}. 
\end{proof}

\begin{uw}\label{horocycle}
A generalized circle orthogonal to the imaginary axis is a horizontal line which corresponds to angle $\beta=\pi$.
It has infinite radius and does not intersect the real axis.

If as in Lemma \ref{circs_geod} $\mathcal C_1$ is a horizontal line then $\mathcal C_2$ must be too. If $\mathcal C_2$ is horizontal then it is disjont with any $\mathcal C_1$.
\end{uw}

\begin{wn} 
If $h=-\cos\beta$ with $\beta \in \left( 0,\frac{\pi}{2}\right]$ is constant mean curvature of a hypersphere then its (constant) distance $\delta$ from the corresponding totally geodesic hypersurface satisfies $\cos\beta =\tanh \delta$. 
\end{wn}

\begin{proof}
After conformal transformation we have by Lemma \ref{circs_geod} and definition of hyperbolic distance in $\Pi^{n,+}$  that
$e^\delta=\cot\frac{\beta}{2}$. Hence
$$\delta =\ln \cot \frac{\beta}{2}=\ln \frac{1+\cos\beta}{\sin\beta}=\ln\frac{1+\cos\beta}{\sqrt{1-\cos^2\beta}}=
{\rm ath} (\cos\beta).$$

In \cite{G} we could find this formula in the equivalent form $\cot\beta =\sinh \delta$.
\end{proof}

Now we are prepared for

\begin{proof} \emph{(of Theorem \ref{route_geod})}
A geodesic sphere in $\mathbb H^n$ cannot serve as a leaf of a codimension $1$  foliation because its interior has nonzero Euler characteristic and cannot be foliated in a tangently to the boundary. Thus the only possible leaves of totally umbilical foliations on $\mathbb H^n$ are generalized hyperspheres and in fact $|h|\le 1$. 

To prove (i) observe theat if some leave $L_{\gamma (t)}$ is a horosphere "centered" at the begin $\gamma (-\infty)$ then (cf. Remark \ref{horocycle}) all preceding leaves must be horospheres of the same "centre".  This argument works in proof of (iii) as well.

We use a conformal transformation of $\Pi^{n,+}$ which is then hyperbolic isometry to put the geodesic $\gamma$ as the $n$--th half--axis $A_{n,+}$ oriented "up". Any generalized sphere representing a generalized hypersphere orthogonal to $\gamma$ has a center on the $A_{n,+}$.

Consider section of $\Pi^{n,+}$ by any $2$--dimensional plane $P$ containing the $n$--th axis and orthogonal to the ideal boundary.  Then $P\cap \Pi^{n,+}$ is isometric to $\Pi^{2,+}$ and $\mathcal F \cap P$
is generalized hypercycle foliation orthogonal to $\gamma =P\cap A_{n,+}$.

In $\Pi^{2,+}$ we paramtrize the geodesic by arc--length $\gamma (t)=e^t$. Fix $t_1$ and use criterion
from Lemma \ref{circs_geod} to avoid leaves intersecting $\gamma$ over $\gamma (t_1)$ to intersect
the leaf $L_{\gamma (t_1)}$. According to Proposition \ref{curv} the mean curvature of any $L_{\gamma (t)}$ equals $-\cos \beta (t)$ which allow to write for any $t_2>t_1$
$$\frac{e^{t_2}}{e^{t_1}}\ge \frac{\cot\frac{\beta(t_2)}{2}}{\cot\frac{\beta(t_1)}{2}}=
\frac{\ \frac{1-h(t_2)}{\sqrt{1-(h(t_2))^2}}\ }{\frac{1-h(t_2)}{\sqrt{1-(h(t_2))^2}}}
=\frac{\left(\sqrt{\frac{1+h(t_2)}{1-h(t_2)}}\right)^{-1}}{\left(\sqrt{\frac{1+h(t_1)}{1-h(t_1)}}\right)^{-1}}
=\frac{e^{-{\rm ath}(h(t_2))}}{e^{-{\rm ath}(h(t_1))}}$$
which is exactly (ii).

Now assume that $h$ is continuous, bounded by $1$ and satisfy (\ref{eq:h_geod}). Conditions (i) and (iii) imply
proper foliation inside last of initial horosphere and first of finishing one. From (ii) we know that generalized hypercycles of given mean curvature are pairwise disjont. Completeness of the foliation comes from continuity of the family in the half--space and the ideal boundary followed by continuity of  $h$ and Lemma \ref{circ_geod}.  
\end{proof}

If we assume that a foliation is transversely differentiable then the condition on umbilical route is even simpler.

\begin{tw}
\label{route_geod_c1}
For function $h$ of mean curvature of leaves of a transversely $C^1$ totally umbilical codimension $1$ foliation of $\mathbb H^n$ along arc--length parametrized geodesic (transversal orienation opposite to the geodesic) there are  $ t_-, t_+\in [-\infty,+\infty]$ such that
\begin{equation}\label{eq:h_c1_geod}
\begin{array}{rl}
({\rm i})   & h|_{(-\infty ,t_-]}\equiv -1, \\
({\rm ii}) & h'\ge h^2-1, \\
({\rm iii})  & h|_{[t_+,+\infty )}\equiv 1.
\end{array}
\end{equation}

Conversely,  if $h:\mathbb R\to [-1,1]$ is a $C^1$--function satisfying 
{\rm (\ref{eq:h_c1_geod})} 
then $h$ is an umibilical route along any geodesic line in $\mathbb H^n$.
\end{tw}

\begin{proof}
It is enough to differentiate (\ref{eq:h_geod}):
$$-\frac{{\rm ath}( h(t+\varepsilon))-{\rm ath}( h(t))}{\varepsilon}\le 1$$ to obtain
$$\frac{h'}{1-h^2}\ge -1.$$
\end{proof}

\begin{uw}
The condition (\ref{eq:h_c1_geod})(ii) on the derivative of $h$ could be formulated in terms of the angle of intersection as
$$\beta ' \ge -\sin\beta.$$
\end{uw}

\begin{prz}
\begin{enumerate}
\item Totally geodesic foliation $h\equiv 0$ represented by concentric spheres.
\item Horospherical foliation $h\equiv -1$ represented by spheres tangent at one point.
\item Pencil foliation $h=-\tanh$ which is ultimate for estimation (\ref{eq:h_c1_geod}) is represented
by spheres having an $(n-2)$--dimensional sphere $\subset \mathbb R^{n-1}\times \{ 0\}$ in common. In dimension $2$ it looks like "rising sun". 

\begin{figure}[h]
\centerline{
\includegraphics[scale=0.7]{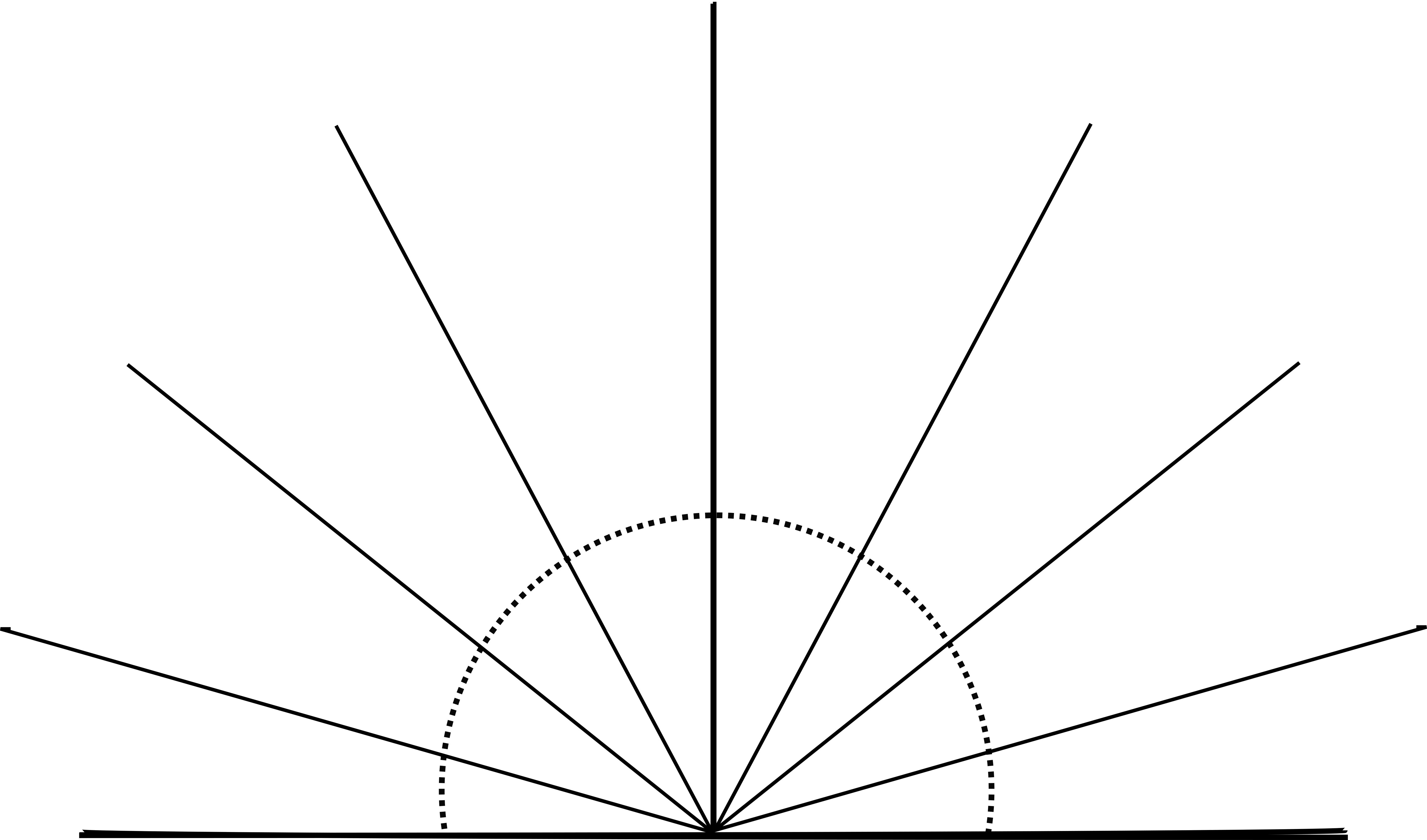}
}
\caption{"Rising sun" foliation in $\mathbb H^2$ of common ends $0$ and $\infty$}
\label{fig:fig2}
\end{figure}

\item For a given (even non--discrete) family  of generalized hypercycles pairwise disjont and orthogonal to a given geodesic we could define a totally umbilical foliation of whole $\mathbb H^n$ with such leaves. 
\end{enumerate}
\end{prz}

\section{Umbilical routes along hypercycles}
Here we change geodesic curvature of transversal an look for a condition for a curve of constant nonzero curvature. Even in this situation formulae looks something strange and do not promise reasonable generalization.

\begin{tw}
\label{route_hyperc}
Let $0<\varphi <\frac{\pi}{2}$.
Assume that $\mathcal F$ is a transversely $C^0$ codimension $1$ totally umbilical foliation of $\mathbb H^n$ orthogonal to an  arc--length parametrized $\varphi$--hypercycle $\gamma$. If for any $t\in\mathbb R$ the mean curvature of the leaf (taken with orientation opposite to $\gamma$) through $\gamma (t)$ is $h(t)$ then $|h|\le \sin \varphi$ 
and there are $ t_-, t_+\in [-\infty,+\infty]$ such that
\begin{equation}\label{eq:h_hyperc}
\begin{array}{rl}
({\rm i})   & h|_{(-\infty ,t_-]}\equiv -\sin\varphi, \\
({\rm ii})  & 
 t\mapsto \ln \dfrac{\sin\varphi -h(t)}{h(t)\cos\varphi+\sqrt{1-(h(t))^2}\,\sin\varphi} \\ 
& {is \ a\  (\sin\varphi)\! -\! {Lipschitz}\ function\ on\ (t_-,t_+)}\\  
\\
({\rm iii})  & h|_{[t_+,+\infty )}\equiv \sin\varphi.
\end{array}
\end{equation}

Conversely, if $h:\mathbb R\to [-\sin \varphi, \sin \varphi]$ is a continuous function satisfying {\rm (\ref{eq:h_hyperc})} then $h$ is an umbilical route along any $\varphi$--hypercycle in $\mathbb H^n$. 
\end{tw}

We shall modify lemmas \ref{circ_geod} and \ref{circs_geod}.
 Denote by $E_{\varphi}$ the open ray $e^{i\varphi}\mathbb R_+ \subset\mathbb C$.

\begin{lem}\label{circ_hyperc}
Let $\varphi \in \left(0,\frac{\pi}{2} \right)$, $s>0$, $\beta\in [0,\pi )$, and $\beta \ne \frac{\pi}{2}+\varphi$. Assume that $\mathcal C\subset\mathbb C$ is a circle of center $C$ and radius $R$ orthogonal to $E_{\varphi}$ at the unique point $se^{i\varphi}$ and meeting the real axis $\mathbb R$ at angle $\beta$.

Then $\beta \in \left[ \frac{\pi}{2}-\varphi , \frac{\pi}{2}+\varphi \right)$, 
$$R=\frac{s\sin\varphi}{\sin\varphi+\cos\beta}, \quad 
C=\frac{s\cos\varphi\cos\beta}{\sin\varphi+\cos\beta}+i\frac{s\sin\varphi \cos \beta}{\sin\varphi+\cos\beta}$$
and the point(s) of the intersection $\mathcal C\cap \mathbb R$ are of the form 
$$a_{\mp}=\frac{s\cos (\varphi \pm \beta)}{\sin\varphi+\cos\beta}.$$
\end{lem}

\begin{proof}
Since $\mathcal C\perp E_{\varphi}$, $C=ce^{i\varphi}$ for some $c\in\mathbb R$. Thus 
$$R=\left|ce^{i\varphi}-se^{i\varphi}\right| = |s-c|=s-c$$
because if $c=s+R$ then $(s+2R)e^{i\varphi}$ would be the second point of intersection $E_{\varphi}\cap \mathcal C$. The top point of $\mathcal C$ is 
$$C+iR=(s-R)e^{i\varphi}+iR=(s-R)\cos\varphi +i (s\sin\varphi + (1-\sin\varphi)R).$$
Using translation of $-\Re C$ we reduce the situation to Lemma \ref{circ_geod}. Now
$$R=\frac{s\sin\varphi +(1-\sin\varphi)R}{1=\cos\beta}$$
hence
$$R=\frac{s\sin\varphi}{\sin\varphi +\cos\beta} \ \ {\rm and}\ \ 
C=(s-R)e^{i\varphi} = \frac{s\cos\beta}{\sin\varphi +\cos\beta}e^{i\varphi}$$

\begin{figure}[h]
\centerline{
\includegraphics[scale=1]{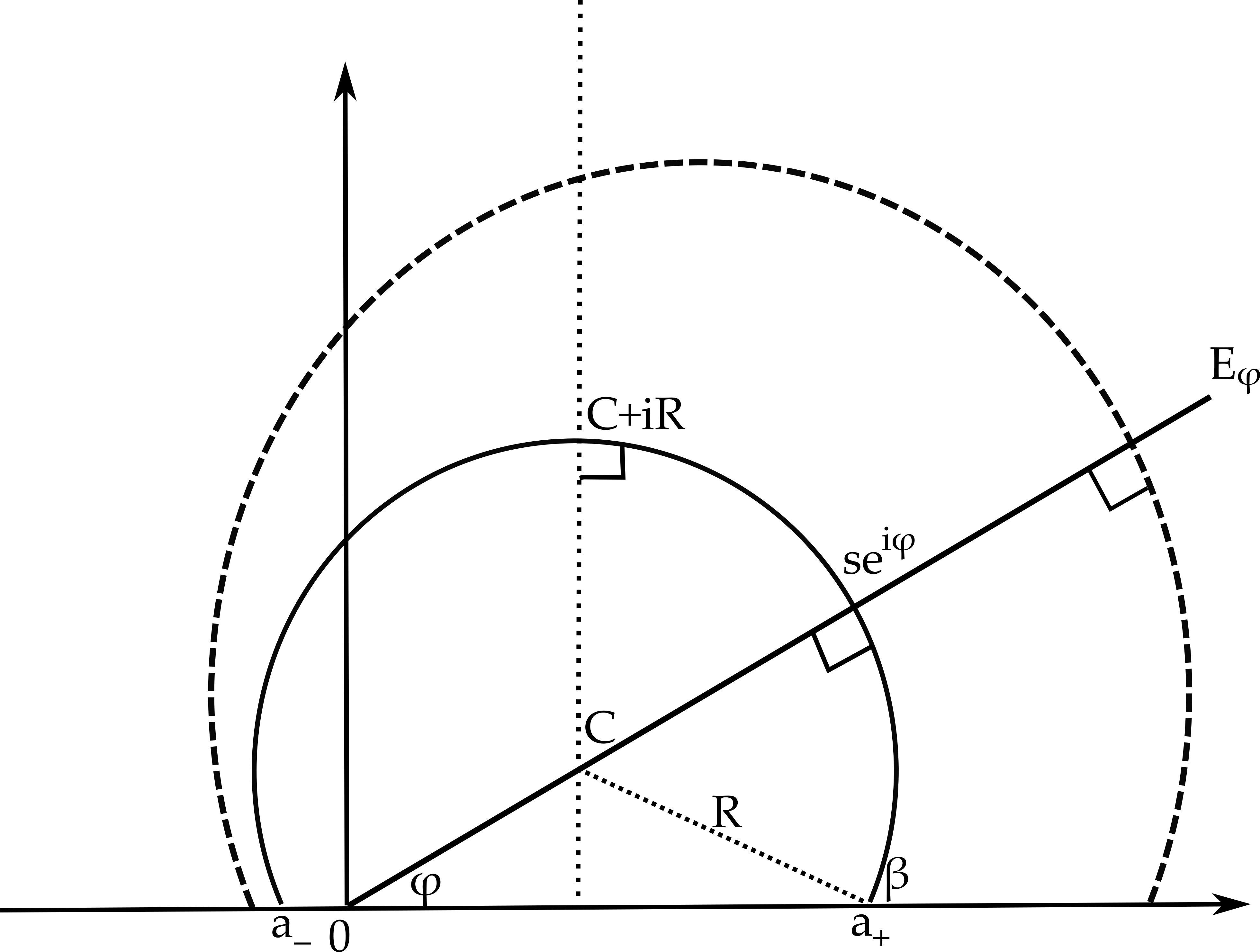}
}
\caption{Hypercycles orthogonal to a hypercycle}
\label{fig:fig3}
\end{figure}

Moreover, $\mathcal C$ intersects the real axis at points
\begin{align*}
\Re C & \pm (\Im C+R)\tan \frac{\beta}{2}=\frac{s\cos\varphi\cos\beta}{\sin\varphi+\cos\beta} \\
& \pm \left(\frac{s\sin\varphi \cos \beta}{\sin\varphi+\cos\beta}+\frac{s\sin\varphi}{\sin\varphi+\cos\beta}\right)
\frac{\sin\beta}{1+\cos\beta} \\
& =\frac{s\cos (\varphi \mp \beta)}{\sin\varphi+\cos\beta}.
\end{align*}

Observe that $\mathcal C$ cuts $E_{\varphi}$ at a unique point iff $|C|\le R$ i.e.
$|\cos\beta |\le \sin\varphi$ so $\frac{\pi}{2}-\varphi \le \beta <\frac{\pi}{2}+\varphi$. But this means that the left and right hand points of $\mathcal C\cap \mathbb R$ are respectively
$$a_-=\frac{s\cos (\varphi + \beta)}{\sin\varphi+\cos\beta}\le 0, 
\quad a_+=\frac{s\cos (\varphi - \beta)}{\sin\varphi+\cos\beta}>0.$$ 

\end{proof}

\begin{lem}\label{circs_hyperc}
Let  $\varphi \in \left(0,\frac{\pi}{2} \right)$, $0<s_1<s_2$,
and $\beta_1, \beta_2 \in \left[\frac{\pi}{2}-\varphi , \frac{\pi}{2}+\varphi\right)$.
Assume that $\mathcal C_1, \mathcal C_2\subset\mathbb C$ are circles  orthogonal to $E_{\varphi}$ at  points $s_1e^{i\varphi}$, $s_2e^{i\varphi}$ and meeting the real axis $\mathbb R$ at angles $\beta_1$, $\beta_2$, respectively.

Then $\mathcal C_1$ and $\mathcal C_2$ do not intersect in the upper half--plane $\Pi^{2,+}$ iff
$$ \frac{\ \ \frac{\sin\varphi+\cos\beta_2}{\cos(\varphi +\beta_2 )}\ \ }{\frac{\sin\varphi+\cos\beta_1}{\cos(\varphi +\beta_1 )}}\le \frac{s_2}{s_1} .$$
\end{lem}

\begin{proof}
Under these assumptions circles $\mathcal C_1$ and $\mathcal C_2$ do not intersect in $\Pi^{2,+}$ iff $\mathcal C_1$ is inside $\mathcal C_2$ (including internal tangency) or they intersect on the side of $-s_1e^{i\varphi}$ not "too high" i.e. upper intersection point is still under or on $\mathbb R$.

By the (Euclidean) symmetry in the hypercycle this is equivalent to the request that the left hand point of $\mathcal C_1\cap \mathbb R$ is to the right of the left hand point of $\mathcal C_2\cap \mathbb R$. Now it is enough to use Lemma \ref{circ_hyperc}.
\end{proof}

\begin{uw}\label{plane_hyperc}
If $\beta =\frac{\pi}{2}+\varphi$ the role of circle meeting $\mathbb R$ at this angle plays a straight
line and their unique common point is $\frac{s}{\cos\beta}$. 

If $\mathcal C_1$ is such a line $\mathcal C_2$ has no room to bend and must a line. For the same reason, circles preceeding a circle meeting $\mathbb R$ at angle $\frac{\pi}{2}-\varphi$ must make the same angle with $\mathbb R$.
\end{uw}

\begin{proof} \emph{(of Theorem \ref{route_hyperc})}
As in the proof of Theorem \ref{route_geod} we conclude that the only possible leaves are generalized hypercycles. Every of them is diffeomorphic to $\mathbb R^{n-1}$ and divides $\mathbb H^n$ into two parts diffeomorphic to $\mathbb R^n$. Suppose that a transversal meets a leaf $L$ twice. Then the transversal of $\mathbb R^n$  by $\mathbb R^{n-1}$ must be tangent to some leaf. This contradiction proves that the $\varphi$--hypercycle meets orthogonally every leaf of the totally umbilical foliation at most once.  

Likely Theorem \ref{route_geod} we reduce the situation to dimension $2$ with $A_{n,+}$ being the geodesic from which the $\varphi$--hypercycle is equidistant. Lemma \ref{circ_hyperc} implies $|h|\le \sin\varphi$. (i) and (iii) are explained in Remark \ref{plane_hyperc}.

To prove (ii) recall that $E_{\varphi}$ has arc--length parametrization $$\gamma (t)=e^{t\sin\varphi +i\varphi\ },\quad t\in\mathbb R$$
an use  Lemma \ref{circs_hyperc}. In fact, for given $t_1<t_2$ we have
$$
\left( e^{t_2-t_1}\right)^{\sin\varphi}=\frac{e^{t_2\sin\varphi}}{e^{t_1\sin\varphi}}
\ge \frac{\ \frac{\sin\varphi -h(t_2)}{\ - h(t_2)\cos\varphi- \sqrt{1-(h(t_2))^2}\, \sin\varphi\ }}{\ \frac{\sin\varphi -h(t_1)}{\ - h(t_1)\cos\varphi- \sqrt{1-(h(t_1))^2}\, \sin\varphi\ }}
$$
which needs only logarithm. 

For the converse, argument from Theorem \ref{route_geod} works similarly but the foliation orthogonal to one hypercycle does not fill all the $\mathbb H^n$. Anyway leaves of such a foliation have a limit (on both sides) which is an umbilical hypersurface. Domains bounded by hypercycles can be easily extended to a umbilical foliation of $\mathbb H^n$ for example adding leaves of the same mean curvature.

\begin{figure}[h]
\centerline{
\includegraphics[scale=0.9]{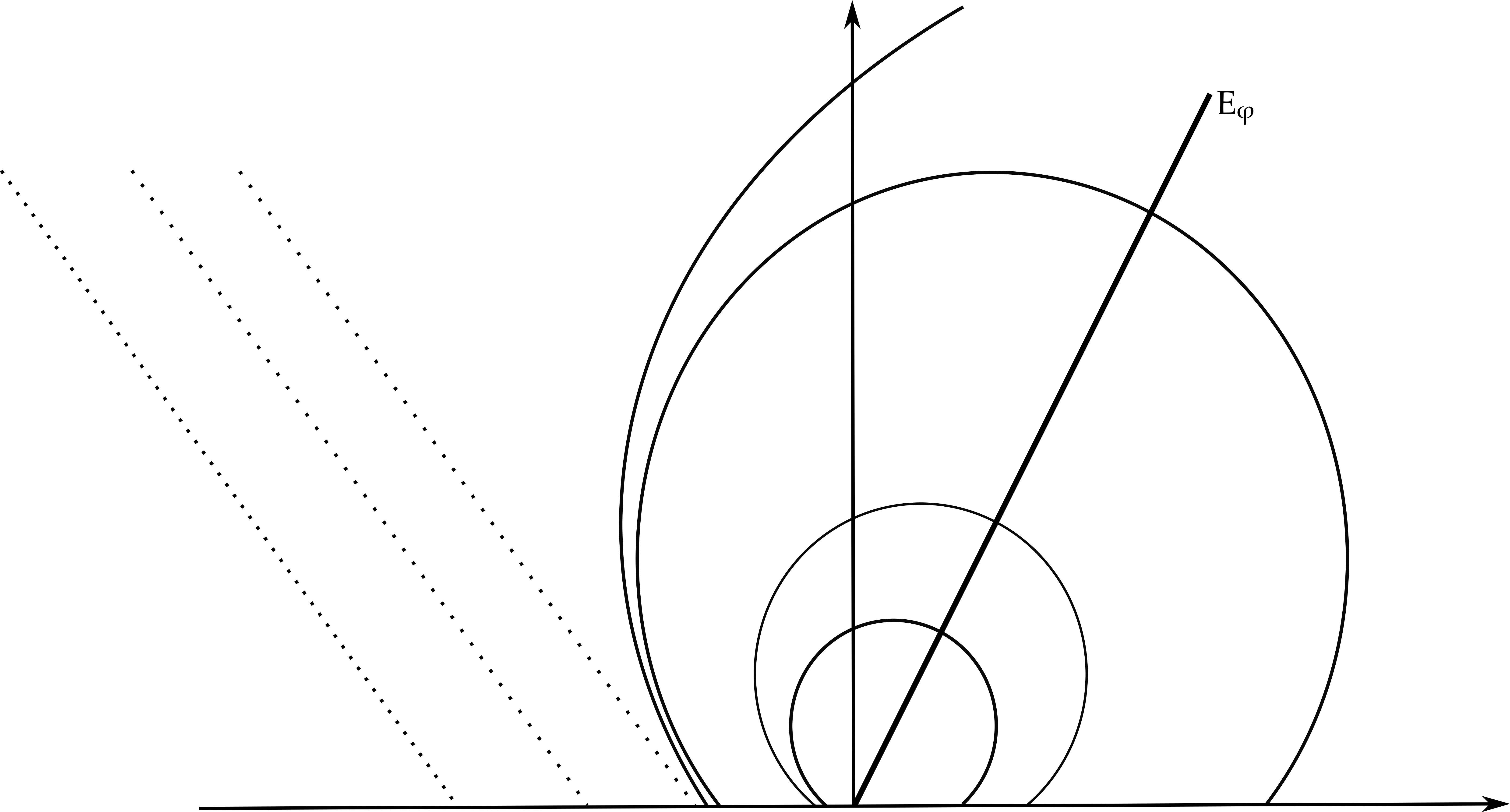}
}
\caption{Extension of umbilical foliation orthogonal to a hypercycle}
\label{fig:fig4}
\end{figure}

\end{proof}

Diffrentiation of (\ref{eq:h_hyperc})(ii) leads to

\begin{tw}
\label{route_hyperc_c1}
The function $h$ of mean curvature of leaves of totally umbilical   transversally $C^1$  codimension $1$ foliation of $\mathbb H^n$ along arc--length parametrized $\varphi$--hypercycle satisfies
$|h|\le \sin \varphi$ 
and there are $ t_-, t_+\in [-\infty,+\infty]$ such that
\begin{equation}\label{eq:h_c1_hyperc}
\begin{array}{rl}
({\rm i})   & h|_{(-\infty ,t_-]}\equiv -\sin\varphi, \\
({\rm ii})  & 
  h'(t)\ge \frac{(h(t) -\sin\varphi)\left(h(t)\cos\varphi+\sqrt{1-(h(t))^2}\,\sin\varphi\right)
\sqrt{1-(h(t))^2}}
{1-h\sin\varphi +\sqrt{1-(h(t))^2}\cos\varphi} \\ 
&  on\ (t_-,t_+)\\  
\\
({\rm iii})  & h|_{[t_+,+\infty )}\equiv \sin\varphi.
\end{array}
\end{equation}

Conversely,  if $h:\mathbb R\to [-\sin\varphi,\sin\varphi]$ is a $C^1$--function satisfying 
{\rm (\ref{eq:h_c1_hyperc})} 
then $h$ is an umbilical route along any geodesic line in $\mathbb H^n$.
\end{tw}

\begin{uw}
The condition from Theorem \ref{route_hyperc_c1} looks shorter in terms of angle of intersection
$$\beta ' \ge \frac{(\sin\varphi+\cos\beta)\, \cos(\varphi+\beta)}{1+\sin(\varphi +\beta)}.$$
Indeed, it is enough to reformulate inequality
$$\left( \frac{\sin\varphi+\cos\beta}{-\cos(\varphi+\beta)}\right)' \le \sin\varphi .$$
\end{uw}

\medskip

Since any horocycle could be transformed into a line parallel to $\mathbb R^{n-1}\times \{ 0\}$, any sphere orthogonal to it intersects the horocycle in two points. The only generalized hyperspheres orthogonal to  the horocycle are $0$-hyperspheres represented by vertical hyperplanes. This motivates the following   
\begin{wn}
The only umbilical route along a horocycle is $h\equiv 0$.
\end{wn}

\begin{prz}
\begin{enumerate}
\item A family of disjoint totally geodesic hypersurfaces orthogonal to a hypercycle 
at any of its point
foliates whole $\mathbb H^n$.
\item Constant curvature foliation with $h\equiv \sin\varphi$ is represented by parallel hyperplanes.
\item Mean curvature of leaves of a foliation by concentric spheres (with the center outside $\Pi^{n,+}$) is orthogonal to some  $\varphi$--hypercycle varies over $(-\sin\varphi ,\sin\varphi)$ along the hypercycle
but on remaining domain of $\mathbb H^n$ could include even horospheres. 
\end{enumerate}
\end{prz}

\section{Final remarks}

\begin{fa}
At any point of a curve in $\mathbb H^n$ of bounded geodesic curvature $|k_g|\le 1$ one can find  a generalized hypercycle in contact of order $2$. Our result give some explanation only for $n=2$.

If a curve exceeds curvature $1$ then umbilical routes disappear --- leaves orthogonal to such a curve intersect even in totally geodesic case like Ferus classification \cite{F}.

\end{fa}

\begin{fa}
If $k_g(p)$ and $k_n(p)$ denote respectively the norm of the second fundamental form of the leave at $p$ and  geodesic curvature of an orthogonal transversal then $k_g^2+k_n^2\le 1$ in case of totally umbilical foliations along hypercycles. 

This is a very special case of  of Hadamard foliations (cf. \cite{Cz2}) for which this estimation is suspected to be true.
\end{fa}

\begin{fa}\label{shadok}
Totally umbilical foliations of $\mathbb H^n$ could be described in a purely conformal way. In fact, the ideal boundary and totally umbilical leaves are represented by generalized spheres and the mean curvature depends only on angle of intersection. 

This provides description of such objects in the space of spheres --- de Sitter space which is quadric in the Lorentz space. The author and Langevin gave in \cite{CzL} a local classification based on boosted time cones and deduce some global facts on curvature of orthogonal transversal.
\end{fa}

\begin{fa}
In the paper we restricted to hypercyclic orthogonal transversals as the most similar to totally umbilical 
higher--dimensional submanifolds.

We could define a \emph{bi--umbilical foliation} as totally umbilical foliation with a/all transversals being totally umbilical. The classification of bi--umbilical foliation on $\mathbb H^n$ may be of some interest even for codimension $2$ in $\mathbb H^4$.   
\end{fa}

\end{document}